\documentclass[11pt]{article}
    
\usepackage{amsmath,amsfonts,amsthm,amscd,amssymb,graphicx}

\usepackage{subfigure}
\usepackage{xcolor}

\numberwithin{equation}{section}

\usepackage{hyperref}

\newtheorem{theorem}{Theorem}[section]

\newtheorem{lemma}[theorem]{Lemma}
\newtheorem{lem}[theorem]{Lemma}
\newtheorem{proposition}[theorem]{Proposition}

\def\eps{\varepsilon }

\newcommand{\bea}{\begin{eqnarray}}
\newcommand{\eea}{\end{eqnarray}}

\newcommand{\RR}{\mathbb{R}}

\newcommand{\ZZ}{{\mathbb Z}}
\newcommand{\TT}{{\mathbb T}}

\def\beq{\begin{equation}}
\def\eeq{\end{equation}}
\def\bb1{{1\!\!1}}

\def\cL{\mathcal{L}}

\newcommand{\pa}{{\partial}}

\def\eps{\varepsilon}

\def\Ff{\widehat{f}}
\def\FK{\widehat{K}}
\def\Fg{\widehat{g}}
\def\Frho{\widehat{\rho}}
\def\FE{\widehat{E}}
\def\FS{\widehat{S}}

\def\Fmu{\widehat{\mu}}

\def\TK{\widetilde{K}}

\begin{document}

\title{Landau damping for analytic and Gevrey data} 

\author{Emmanuel Grenier\footnotemark[1]
  \and Toan T. Nguyen\footnotemark[2]
\and Igor Rodnianski\footnotemark[3]
}

\maketitle

\renewcommand{\thefootnote}{\fnsymbol{footnote}}

\footnotetext[1]{CNRS et \'Ecole Normale Sup\'erieure de Lyon, Equipe Projet Inria NUMED,
 INRIA Rh\^one Alpes, Unit\'e de Math\'ematiques Pures et Appliqu\'ees., 
 UMR 5669, 46, all\'ee d'Italie, 69364 Lyon Cedex 07, France. Email: Emmanuel.Grenier@ens-lyon.fr}

\footnotetext[2]{Penn State University, Department of Mathematics, State College, PA 16803. Email: nguyen@math.psu.edu. TN is a Visiting Fellow at Department of Mathematics, Princeton University, and partly supported by the NSF under grant DMS-1764119, an AMS Centennial fellowship, and a Simons fellowship. }

\footnotetext[3]{Princeton University, Department of Mathematics, Fine Hall, Washington Road, Princeton, NJ 08544. Email: irod@math.princeton.edu. IR is partially supported by the NSF 
grant DMS \#1709270 and a Simons Investigator Award. }

\subsubsection*{Abstract}

In this paper, we give an elementary proof of the nonlinear Landau damping for the Vlasov-Poisson system near Penrose stable equilibria on the torus $\TT^d \times \RR^d$ that was first obtained by Mouhot and Villani in 
\cite{MV} for analytic data and subsequently extended by Bedrossian, Masmoudi, and Mouhot \cite{BMM} for Gevrey-$\gamma$ data, $\gamma\in(\frac13,1]$. 
Our proof relies on simple pointwise resolvent estimates and a standard nonlinear bootstrap analysis, using an
ad-hoc family of analytic and Gevrey-$\gamma$ norms.



\section{Introduction}


We study the classical Vlasov-Poisson system 
\begin{equation}
\label{VP}
\partial_t f + v \cdot \nabla_x f + E \cdot \nabla_v f = 0, \qquad \nabla_x \cdot E = \rho - 1
\end{equation} 
for $x \in \TT^d$ and $v\in \RR^d$, where $\rho(t,x)$ denotes the charge density 
\begin{equation}
\label{def-rho}
\rho(t,x) = \int_{\RR^d} f(t,x,v)\; dv. 
\end{equation} 
The problem \eqref{VP}-\eqref{def-rho} will be considered for analytic or Gevrey initial data $f^0(x,v)$, satisfying
\beq \label{totalmass}
\iint_{\TT^d \times \RR^d} f^0(x,v)\; dxdv = 1.
\eeq
Note that the total mass is conserved in time: $\iint f(t,x,v)\;dxdv =1$ for all times, and thus the Poisson equation
\footnotetext{The results also apply to the gravitational case $E=\nabla\phi$ with the appropriate modification of the Penrose condition, where $\mu$ is replaced by $-\mu$.}

$$
-\Delta\phi=\rho-1,\qquad E= -\nabla\phi
$$
is solvable on $\TT^d$. 

In this paper, we study the large time behavior of solutions near a homogenous equilibrium 
$$
 f = \mu(v) , \qquad E = 0,
$$
with $\int \mu(v)\; dv = 1$. This leads to the following perturbed nonlinear Vlasov-Poisson system  
\begin{equation}
\label{VP-pert}
\partial_t f + v \cdot \nabla_x f + E \cdot \nabla_v \mu  = - E \cdot\nabla_v f, \qquad \nabla \cdot E = \rho, 
\end{equation} 
with initial perturbations  $f^0(x,v)$ of zero mass. 

Throughout this paper, we consider equilibria $\mu(v)$ so that 

\begin{itemize}

\item $\mu(v)$ is real analytic and satisfies 
\begin{equation}
\label{reg-mu}
|\widehat{\langle v\rangle^2\mu}(\eta)| + |\widehat{\mu}(\eta)|  \le C_0 e^{-\theta_0 |\eta|},
\end{equation} 
where $\widehat{\mu}$ denotes the Fourier transform of $\mu(v)$. 

\item $\mu(v)$ satisfies the Penrose stability condition, namely, 
\begin{equation}\label{Penrose0}
\inf_{k\in \ZZ^d  \setminus\{0\}; \Re \lambda \ge 0} \Big |1 + \int_0^\infty e^{-\lambda t} t \widehat{\mu}(kt) \; dt \Big| \ge \kappa_0 >0 .
\end{equation}
		
\end{itemize}

The Penrose stability condition holds for a variety of equilibria including the Gaussian $\mu(v) = e^{-|v|^2/2}$. 
In three or higher dimensions, the condition is valid for any positive and radially symmetric equilibria \cite{MV}. 

For the above class of so-called Penrose stable equilibria, we study {\em Landau damping} -- decay of the electric field for large times. 
The linear Landau damping was discovered and fully understood by Landau \cite{Landau}. 
The nonlinear damping has been proved in the case of analytic data by Mouhot and Villani 
in their celebrated work \cite{MV}. Their proof has then been simplified, and the result has been extended for Gevrey data, in \cite{BMM}. 

In this paper, we will give an elementary and short proof of the aforementioned results: namely, if initial perturbations $f^0$ are sufficiently small in a suitable analytic or Gevrey regularity space, then the nonlinear Landau damping occurs. 
More precisely, (with $k\in \Bbb T^d$ and $\eta\in \Bbb R^d$ denoting the Fourier variables conjugate to 
$x$ and $v$, respectively,  and $\langle k,\eta\rangle=\sqrt{1+|k|^2+|\eta|^2}$, $\hat f_{k,\eta}=\hat f(k,\eta)$)

\begin{theorem}\label{theo-LD}
Let $\lambda_1 > 0$ and $\gamma\in (\frac13,1]$. There exists $\eps$ such that for any initial data $f^0$ satisfying
$$ 
\sum_{j=0}^{d} \sum_{k\in \ZZ^d}  \int_{\RR^d}  e^{2\lambda_1 \langle k,\eta\rangle^\gamma} 
|\partial^j_\eta \widehat f^0_{k,\eta}|^2 d \eta  \le \eps,
$$
the nonlinear Landau damping occurs: that is, $E(t,x)$ goes to $0$ exponentially fast in $\langle t\rangle^\gamma$, and $f(t,x+vt,v)$ converges to a limit $f_\infty(x,v)$.
\end{theorem}

As mentioned, the result has already been proved by Mouhot and Villani in \cite{MV} for analytic data ($\gamma=1$) and by Bedrossian, Masmoudi, and Mouhot \cite{BMM} ($\gamma>1/3$). Here we propose a simpler proof of this classical result. The simplification comes from the observations that

\begin{itemize}

\item the result is a linear perturbation of the same result for the free transport, the latter result being obvious.
As a matter of fact, a solution $f(t,x,v)$ of the free transport is simply given by
\beq \label{free}
f(t,x,v) = f^0(x - vt, v) .
\eeq
If $f^0$ is smooth, then the corresponding electric field $E$ goes to $0$. If the regularity of $f^0$ is in a Sobolev class, then
$E$ goes to $0$ at a polynomial rate. If it is analytic or Gevrey, then $E$ goes to $0$ exponentially fast.

\item echos are suppressed. Echo in plasma refers to the interaction of two waves
$f_j e^{i k_j (x-vt) + i \eta_j v}$ for $j = 1, 2$. These two waves interact, through a nonlinear term of the form 
$E_1\pa_v f_2$, and give birth to a third one. The latter is the wave 
$$
i  \delta_{\eta_1 = k_1t} f_1 f_2 (\eta_2-k_2 t) 
\left\{e^{ik_1 x + k_2 (x-vt) + i \eta_2 v}- e^{i(k_1 + k_2) (x-vt) + i \eta_2 v}\right\}.
$$
The electric field of this third wave may become important later,  at very large times. However, echoes can
appear at large times only if the initial data has components of high spatial frequencies, which is not the
case with analytic and Gevrey initial data (see \cite{Bedrossian,GNR2} for more details). This is manifest already 
in the presence of the term $\delta_{\eta_1 = k_1t}$, which does not vanish iff $\eta_1=k_1t$, as well as the 
additional factors of either $\delta_{\eta_2 =k_2t}$ or $\delta_{\eta_2 = (k_1+k_2)t}$, which will 
appear in the expression for the electric field.

\item free transport (\ref{free}) creates large gradients in $v$.
More precisely derivatives in $v$ growth polynomially in time. An easy way to fight this growth
is to have an exponentially decreasing electric field, which is the case for analytic initial data.

\item choice of the analytic norms. We build a norm which 
controls the derivatives of $f$ at any order. 
We consider all regularities at
the same time, using so called ``generator functions'', and vary 
 the analyticity radius, which gives us extra flexibility in the control on the electric field.

\end{itemize}
We note that, while nonlinear Landau damping successfully describes long time behavior of solutions the Vlasov-Poisson 
equations near stable Penrose equilibria for {\it analytic/Gevrey} data, the corresponding problem for perturbations 
(and equilibria) of finite and even ${\mathcal{C}}^\infty$ regularity is open, see however \cite{Bedrossian,GNR2}.

The analytic and Gevrey norms are introduced in Section $2$. Section $3$ is devoted to a review of the classical proof
of linear Landau damping. The nonlinear Landau damping is proved in Section $4$. 
For the sake of simplicity, we detail the proof in one space dimension $d = 1$. The arguments in
higher dimensions follow with obvious modifications.


\section{Analytic and Gevrey spaces}


We construct solutions in analytic and Gevrey spaces. First, following the characteristics of the free transport, we introduce
$$
g(t,x,v) = f(t,x + v t, v).
$$
Note that
$$
\partial_v g = t \partial_x f + \partial_v f, \quad \partial_x g = \partial_x f.
$$
This transforms \eqref{VP-pert} into 
\begin{equation}\label{VP-g} 
\partial_t g + E(t,x+vt)\partial_v \mu(v)= - E(t,x+vt) (\partial_v - t \partial_x) g 
\end{equation}
where $E(t,x)$ solves the Poisson equation \eqref{VP}. Note that the density $\rho$ satisfies 
\begin{equation}\label{new-rho} 
\rho(t,x) = \int_{\RR} g(t,x-vt, v)\; dv.
\end{equation}
We measure the analyticity or Gevrey regularity of solutions using the following "generator functions". Namely
for $z\ge 0$, $\gamma>0$, and $\sigma>0$, we introduce
$$ 
G[g](z) := \sum_{k\in \ZZ}  \int_\RR e^{2z\langle k,\eta\rangle^\gamma} 
\Bigl[ |\widehat g_{k,\eta}|^2 + |\partial_\eta \widehat g_{k,\eta}|^2 \Bigr] \langle k,\eta\rangle^{2\sigma} d \eta
$$
and {for $\alpha<\frac12$,}
\begin{equation}\label{def-Nrho}
F[\rho](t,z) :=  \sup_{k\in \ZZ\setminus \{0\}}  e^{z \langle k, kt \rangle^\gamma } |\widehat \rho_k(t)| 
\langle k,kt\rangle^{\sigma} { |k|^{-\alpha}}.
\end{equation}
The parameter $z$ is the analyticity radius. We will eventually use these norms with a time-dependent $z(t)$ 
of the form $z(t)=\lambda_0(1+(1+t)^{-\delta})$, slowly decaying in time.

Let us  comment on these two norms.
First, these norms are classical and are very close to those used in \cite{MV,BMM}, with $\gamma=1$ corresponding to the analytic regularity. However, in \cite{MV}, the radius of analyticity -- parameter
$z$ -- it not time-dependent and the argument involves a Nash-Moser iteration scheme. On the other hand, \cite{BMM} makes use of of a flexibility of the analyticity radius to avoid the Nash-Moser process but 
relies on the space-time $L^2$ estimates for the inverse of the Penrose operator. We replace this by classical pointwise estimates on the resolvent. 
 
 Note also that using the Sobolev weight $\langle \eta \rangle^\sigma$ or $\langle k t \rangle^\sigma$
 is equivalent to working with analytic norms on derivatives of $g$ with respect to $v$ or $t$. For the Vlasov-Poisson system, when we consider $E \partial_v f$ as a source term, we loose one derivative (in $v$) and one factor $t$, which
 also can be seen as another loss of derivative. However if we differentiate $f$ with respect to $v$ and $E$ with
 respect to $t$, then we are back to a  loss of one derivative only, for the system posed for $f$, $\partial_v f$
 and $\partial_x E$. This very classical trick of "quasilinearization" leads to the introduction of the weights
 $\langle \eta \rangle^\sigma$ and $\langle k t \rangle^\sigma$.
   
 Adding this Sobolev weight also greatly improves behavior of the norms with respect to 
 "paraproducts". More precisely, for large $\eta$ and $\eta'$,
 \beq \label{para}
 {  \langle \eta \rangle^\sigma \langle \eta' \rangle^\sigma \over \langle \eta + \eta' \rangle^\sigma }
 \ge C_0 \min \Bigl( \langle \eta \rangle, \langle \eta' \rangle \Bigr)^\sigma,
 \eeq
 hence we gain a factor $\langle \min(\eta,\eta') \rangle^\sigma$ during multiplications. As $\sigma$
 can be large, this compensates a very small analyticity radius decay and produces
 a large extra decay for the electric field; see Proposition \ref{prop-rho}. This extra decay is not present when $\sigma = 0$.

 Note that the weight $e^{z \langle k, k t\rangle^\gamma}$ in the norm on $\rho$ already encodes the exponential
 decay of the electric field with time for analytic or Gevrey initial data.

 Let us now recall some properties of generator functions. 
 They are non negative, and all their derivatives are non negative and non decreasing in $z$.  
Moreover, generator functions have properties with respect to algebraic operations and differentiation. 
Namely they ``commute'' with products, sums and derivatives, making them a very versatile tool for existence and stability/instability results \cite{GrN6, GrN8}. 
For instance, for any function $f$,
\begin{equation}\label{dervG}
G[\partial_x^{\gamma/2} f] \le \partial_z G[ f ], \qquad G[\partial_v^{\gamma/2} f] \le \partial_z G[f] .
\end{equation}
Note that generator functions are functions of a "regularity index" $z$.
As a consequence, they allow to estimate all the analytic norms (depending on $z$) at the same
time. The evolution of these norms are turned into a simple non linear transport equation on $G$, allowing
a simple control on all the possible analytic norms all at once.


\section{Linear Landau damping}\label{sec-resolvent}


In this section we review the proof of linear Landau damping in the context of  norms based on the generator functions. The study via resolvent estimates is classical; see, for instance, \cite{Degond, Glassey, Lin,HKNR2}.


\subsection{Equation on the density}


In this section, we study the linearized Vlasov-Poisson system around $\mu(v)$, namely, the following linear problem 
\begin{equation}
\label{VP-lin}
\partial_t g + E(t,x+vt)\partial_v \mu  = 0, \qquad \partial_x E = \rho , 
\end{equation} 
with initial data $f^0(x,v)$. 
To solve \eqref{VP-lin}, we follow the standard strategy and first derive a closed equation on the electric field. 
Let $\Frho_k(t)$ be the Fourier transform of $\rho(t,x)$ in $x$, 
and $\Fg_{k,\eta}(t)$ the Fourier transform of $g(t,x,v)$ in $x$ and $v$. Note that as $\Frho_0(t)=0$ for all times, 
throughout this section, we shall only focus on the case when $k \not =0$. We have the following Lemma.

\begin{lemma}
Let $g$ be the unique solution to the linear problem \eqref{VP-lin}. There holds the following closed equation on the density
\begin{equation}
\label{eqs-rho0} \Frho_k(t) + \int_0^t (t-s) \widehat{\mu}(k(t-s)) \Frho_k(s) \; ds = \FS_k(t)
\end{equation}
with the source term 
$$
\FS_k(t) =  \hat f^0_{k,kt}.
$$
\end{lemma}
\begin{proof} Taking the Fourier transform of \eqref{VP-lin} in both $x$ and $v$, we obtain 
$$ 
\partial_t \Fg_{k,\eta} + \FE_k(t) \widehat{\partial_v \mu}(\eta - kt) = 0
$$
since
$$
\iint e^{-i k x - i \eta v} E(t,x+vt) \partial_v \mu(v) dx dt = \FE_k(t) \widehat{\partial_v \mu}(\eta - kt) .
$$
The Lemma follows, upon noting that $\Frho_k(t) = \Fg_{k,kt}(t)$. 
\end{proof}


\subsection{Penrose stability}


In this section we introduce the Penrose condition in order to solve \eqref{eqs-rho0}. For any function $F$ in $L^2 (\RR_+)$, 
we recall that the Laplace transform of $F(t)$ is defined by 
$$ 
\cL[ F](\lambda) = \int_0^\infty e^{-\lambda t} F(t)\; dt 
$$
which is well-defined for any complex value $\lambda$ with $\Re \lambda >0$. 
Taking the Laplace transform of \eqref{eqs-rho0}, we  get 
\begin{equation}\label{Lap-rho} 
\cL[\Frho_k](\lambda) = \frac{\cL[\FS_k](\lambda)}{1 + \cL[t \widehat{\mu}(kt)](\lambda)} .
\end{equation}
The Penrose stability condition ensures that the symbol $1 + \cL[t \widehat{\mu}(kt)](\lambda)$ never vanishes. 
Precisely, we assume that
\begin{equation}\label{Penrose}
\inf_{k\in \ZZ; \Re \lambda\ge 0} |1 + \cL[t \widehat{\mu}(kt)](\lambda)| \ge \kappa_0 
\end{equation}
for some positive constant $\kappa_0$. Writing out the Laplace transform, the above gives \eqref{Penrose0}. 

Using the Penrose stability condition \eqref{Penrose}, Mouhot and Villani \cite{MV} 
obtained the boundedness of the density $\rho(t,x)$ in $L^2_{x,t}$ in term of the source $S(t,x)$. 
This estimate also play a role in \cite{BMM}. In the next section, we will derive {\it pointwise} bounds directly on the resolvent kernel 
via the Laplace-Fourier transform approach. 
The analysis is classical; see, e.g., Degond \cite{Degond} and Glassey and Schaeffer \cite{Glassey}. 
See also \cite{HKNR2} where the pointwise dispersive estimates are obtained for the resolvent kernel on the whole space.


\subsection{Resolvent estimates}


From \eqref{Lap-rho}, we can write 
\begin{equation}
\label{eqs-rholambda}\cL[\Frho_k](\lambda) = \cL[\FS_k](\lambda) +  \TK_k(\lambda) \cL[\FS_k](\lambda)
\end{equation} 
where we denote 
\begin{equation}\label{def-Glambda}
\TK_k(\lambda):= - \frac{\cL[t \widehat{\mu}(kt)](\lambda)}{1 +\cL[t \widehat{\mu}(kt)](\lambda)} .
\end{equation}
Thus, in order to derive pointwise estimates for $\Frho_k(t)$, we first derive bounds on the resolvent kernel 
$\TK_k(\lambda)$.

\begin{lemma}\label{lem-resolvent} 
Assume the Penrose condition \eqref{Penrose}. 
There is a positive constant $\theta_1< \theta_0$ so that the function $\TK_k(\lambda)$ is an analytic function in 
$\{\Re\lambda \ge -\theta_1|k| \}$. In addition, there is a universal constant $C_1$ such that 
\begin{equation}
\label{bd-resolvent}
 |\TK_k(\lambda)| \le {C_1 \over 1 + |k|^2 + |\Im \lambda |^2},
\end{equation}	
uniformly in $\lambda$ and $k\not =0$ so that $\Re \lambda = -\theta_1|k|$. 
\end{lemma}

\begin{proof} 	
We first give estimates on $\cL[t \widehat{\mu}(kt)](\lambda)$. By definition, the Laplace transform 
$$ 
\cL[t \widehat{\mu}(kt)](\lambda) = \int_0^\infty e^{-\lambda t} t \Fmu(k t) \; dt 
$$
is well-defined for $\Re \lambda \gg 1$. In addition, using the assumption \eqref{reg-mu}, 
we in fact have for $\Re \lambda \ge -\theta_1 |k|$ and any $\theta_1\le \frac12\theta_0$,  
\begin{equation}\label{bd-Lmu} 
|\cL[t \widehat{\mu}(kt)](\lambda)| \le C_0 \int_0^\infty e^{-\Re \lambda t} t e^{-\theta_0 |kt|} \; dt \le C_1|k|^{-2}
\end{equation}
for some constant $C_1$ independent of $k$ and $\theta_1$. 
This implies that $\TK_k(\lambda)$ is a meromorphic function in $\{\Re\lambda \ge -\theta_1 |k| \}$ for each $k \not =0$. 
	
In addition, integrating by parts in time, we get  
$$
\begin{aligned}  
\cL[t \widehat{\mu}(kt)](\lambda) &= \int_0^\infty \frac{(M^2_k-\partial_t^2)(e^{-\lambda t})}{M^2_k-\lambda^2} t \Fmu(k t) \; dt 
\\&= \int_0^\infty \frac{e^{-\lambda t}}{M^2_k-\lambda^2} (M^2_k-\partial_t^2)(t \Fmu(k t)) \; dt - \frac{\Fmu (0)}{M^2_k - \lambda^2} 
\end{aligned}$$
for any constant $M_k$ and for any $\lambda \not =M_k$. 
Note that 
$$
\Fmu(0)= \int \mu(v)\; dv= 1.
$$ 
Hence, for $\Re \lambda = - \theta_1|k|$ and for $M_k = 2\theta_1|k|$, we have 
$$
\begin{aligned}  
| \cL[t \widehat{\mu}(kt)](\lambda)| &\le C_0\int_0^\infty \frac{e^{\theta_1 |kt|}}{\theta_1^2 |k|^2+ |\Im \lambda|^2} (|k| + |k|^2 t) 
e^{-\theta_0 |kt |} \; dt + \frac{1}{\theta_1^2|k|^2 + |\Im \lambda|^2}
\end{aligned}$$
As $0\le \theta_1 \le \frac12 \theta_0$ and $k \not =0$, this yields 
\begin{equation} \label{bd-Lmu1}
\begin{aligned}  
| \cL[t \widehat{\mu}(kt)](\lambda)| 
\le C_1 (1 + |k|^2 + |\Im \lambda |^2)^{-1}, 
\end{aligned}\end{equation}
for any $\lambda$ on the line $\{\Re \lambda = -\theta_1 |k|\}$. 
Here, we stress that the constant $C_1$ depends only on $\mu$ through the assumption \eqref{reg-mu}, 
but is independent of $k$ and $\theta_1 \in [0,\frac12 \theta_0]$. 
	
Let us turn to $\TK_k$.	
The estimate \eqref{bd-Lmu}  shows that there is a $k_0$ so that $|\cL[t \widehat{\mu}(kt)](\lambda)| \le \frac12$ 
for all $|k|\ge k_0$. That is, $\TK_k(\lambda)$ is in fact analytic in $\{\Re\lambda \ge -\theta_1 |k| \}$, and satisfies 
$$ 
| \TK_k(\lambda)| \le 2 C_1 |k|^{-2} 
$$
for $|k|\ge k_0$ and for $\Re\lambda \ge -\theta_1 |k|$. 
In addition, for $\Re \lambda = -\theta_1 |k|$, with $\theta_1 \in [0,\frac12 \theta_0]$, using \eqref{bd-Lmu1}, 
we have 
\begin{equation} \label{bd-Gmu1}
\begin{aligned}  
|  \TK_k(\lambda)| 
\le 2C_1 (1 + |k|^2 + |\Im \lambda |^2)^{-1} 
\end{aligned}\end{equation}for $|k|\ge k_0$. 	
	
We next turn to the case when $1\le |k|\le k_0$. In view of \eqref{bd-Lmu1}, there is a constant $\tau_0$ so that 
$$
\cL[t \widehat{\mu}(kt)](\lambda)| \le \frac12
$$	
for all $\lambda$ so that $\Re \lambda = -\theta_1|k|$ and $|\Im \lambda| \ge \tau_0$. 
This proves that $1+\cL[t \widehat{\mu}(kt)](\lambda)$ never vanishes in this range of $\lambda$, 
which again yields \eqref{bd-Gmu1}. 

For $|\Im\lambda|\le \tau_0$ and $|k|\le k_0$, 
as the Penrose condition holds for $\Re \lambda =0$, there is a small positive constant $\theta_1$ so that 
$$
|1 + \cL[t \widehat{\mu}(kt)](\lambda)| \ge \frac12 \kappa_0 
$$	
for $\Re \lambda \ge -\theta_1|k|$ (recalling that $1\le |k|\le k_0$). 
Hence, the bounds on $ \TK_k(\lambda)$ follow from those on $\cL[t \widehat{\mu}(kt)](\lambda)$. 
The Lemma follows. 
\end{proof}


\subsection{Pointwise estimates}


In this section, we prove the following Proposition which gives pointwise estimates on solutions $\Frho_k(t)$ 
to the linear problem \eqref{eqs-rho0}. 

\begin{proposition}\label{prop-exprho} 
Assume the Penrose condition \eqref{Penrose}. The unique solution $\Frho_k(t)$ to the linear problem \eqref{eqs-rho0} 
can be expressed by 
\begin{equation}\label{exp-rho} 
\Frho_k(t) = \FS_k(t) + \int_0^t \FK_k(t-s) \FS_k(s)\; ds 
\end{equation}
where the kernel $\FK_k(t)$ satisfies 
$$ 
|\FK_k(t)| \le C_1 e^{- \theta_1 |kt|} 
$$
for some positive constants $\theta_1$ and $C_1$. 
\end{proposition}

Note that, as $f^0$ is analytic, $\hat S_k(t)=\hat f^0_{k,kt}$ goes exponentially fast to $0$, and thus, using (\ref{VP-lin}),
do does $\rho_k$.

\begin{proof}
 The representation \eqref{exp-rho} follows directly from \eqref{eqs-rholambda}, upon taking the inverse Laplace transform, 
 where the kernel $\FK_k(t)$ is the inverse Laplace transform of the resolvent kernel $\TK_k(\lambda)$ constructed 
 in Lemma \ref{lem-resolvent}. Precisely, we have 
$$
\FK_k(t) = \frac{1}{2\pi i} \int_{\{ \Re \lambda = \gamma_0\}} e^{\lambda t} \TK_k(\lambda)\; d\lambda
$$ 
for some large positive constant $\gamma_0$. By Lemma \ref{lem-resolvent}, 
$\TK_k(\lambda)$ is analytic in $\{\Re\lambda \ge -\theta_1|k| \}$, 
and thus we can deform the complex contour of integration from 
$\{ \Re \lambda = \gamma_0\}$ into $\{ \Re \lambda = - \theta_1|k|\}$, 
on which the resolvent estimate \eqref{bd-resolvent} holds. Therefore,  
$$
|\FK_k(t)| \le C_0 \int_{\{ \Re \lambda = -\theta_1 |k|\}} e^{-\theta_1|kt|} (1 + |k|^2 + |\Im \lambda|^2)^{-1}\; d\lambda 
\le C_1 e^{-\theta_1|kt|}.$$ 
The Proposition follows. 
\end{proof}


\subsection{Gevrey estimates}


We now derive Gevrey estimates on $\rho$.

\begin{lemma}\label{lem-analyticrho} 
Assume the Penrose condition \eqref{Penrose}. 
Let $\Frho_k(t)$ be the unique solution to the linear problem \eqref{eqs-rho0}, and let $F[\rho](t,z)$ 
be the corresponding generator function \eqref{def-Nrho}. Then, there holds  
\begin{equation}\label{bd-propagator} 
F[\rho](t,z) \le F[S](t,z) + C\int_0^t e^{-\frac14\theta_1 (t-s)}F[S](s,z)\; ds 
\end{equation}	
for $t\ge 0$ and for $z\in [0,\frac{\theta_1}{2}]$, with $\theta_1$ defined as in Proposition \ref{prop-exprho}. 
\end{lemma}
Using the definition of $F[S]$ and the fact that $f^0$ is Gevrey, we obtain that $F[S](t,z)$ is uniformly bounded
in time, provided $z$ is small enough. As a consequence, $F[\rho]$ is uniformly bounded in time. By definition of
$F[\rho]$, this leads to an exponential decay of $\FE_k(t)$ for $k \ne 0$.

\begin{proof} In view of the definition \eqref{def-Nrho} and the expression \eqref{exp-rho}, for $k\not =0$, we compute 
$$	
\begin{aligned}
& e^{ z \langle k, kt \rangle^\gamma } |\widehat \rho_k(t)| 
\langle k,kt\rangle^{\sigma} 
\\
&\le e^{ z \langle k, kt \rangle^\gamma }
\langle k,kt\rangle^{\sigma} \Big[  |\FS_k(t)| + C_1 \int_0^t e^{-\theta_1|k|(t-s)} |\FS_k(s)| \; ds\Big].
\end{aligned}$$	
It is sufficient to treat the time integral term, since the other term is exactly $F[S](t,z)$. 
As $z\in [0,\frac12 \theta_1]$ and $0<\gamma\le 1$, we have 
$$ 
e^{z\langle k, kt\rangle^\gamma } e^{-\frac12 \theta_1 |k|(t-s)} \le e^{z \langle k, ks\rangle^\gamma} .
$$ 
Similarly, by considering $s\in [0,t/2]$ and $s\in [t/2,t]$, we get 
$$ 
\langle k, kt \rangle^\sigma e^{-\frac14\theta_1|k|(t-s)} \le C\langle k,ks\rangle^\sigma
$$
for some universal constant $C$ that is independent of $k$ and $t$. 
The Lemma follows, upon noting $|k|\ge 1$. 
\end{proof}



\section{Nonlinear Landau damping}\label{sec-MV}


We now turn to the proof of the nonlinear Landau damping.


\subsection{Bounds on $g$}


We first prove the following 

\begin{proposition}\label{prop-G} 
There is a constant $C_0$ so that 
the following differential inequality holds
\begin{equation}\label{FG1}
\partial_t G[g(t)](z) \le C_0  F[\rho](t,z)  G[g(t)]^{1/2}(z) + C_0(1+t) F[\rho](t,z)  \partial_z G[g(t)](z).
\end{equation}
\end{proposition}


\begin{proof} For convenience, set 
$$A_{k,\eta} = e^{z\langle k,\eta\rangle^\gamma} \langle k,\eta \rangle^\sigma$$
be the symbol of a Fourier multiplier operator $A$. Occasionally, we write $\widehat{Ag}_{k,\eta} = A_{k,\eta}\Fg_{k,\eta}$. Using the triangle inequalities
\begin{equation}\label{basic}
\begin{aligned}
\langle k,\eta \rangle &\le 2\langle k', \eta' \rangle \langle k-k', \eta - \eta' \rangle, \\
\langle k, \eta \rangle &\le \langle k', \eta' \rangle + \langle k - k', \eta - \eta' \rangle ,
\end{aligned}
\end{equation}
we have 
\begin{equation}
\label{basic-A} A_{k,\eta} \le C_0 A_{k',\eta'} A_{k-k', \eta - \eta'},
\end{equation}
for some universal constant $C_0$. We shall use the inequalities \eqref{basic} and \eqref{basic-A} repeatedly throughout the proof. 
From the definition, we compute 
\begin{equation}\label{key-boundGg}
\begin{aligned}
&\partial_t G[g(t)](z) 
\le \sum_{k\in \ZZ}  \int_\RR \Bigl( \partial_t|\widehat g_{k,\eta}|^2 + \partial_t|\partial_\eta \widehat g_{k,\eta}| ^2\Bigr) A_{k,\eta}^2 d \eta .
\end{aligned} \end{equation}
To bound the integral, we take the Fourier transform of the equation \eqref{VP-g}, yielding 
\begin{equation}\label{eqs-g} 
\partial_t \Fg_{k,\eta} 
= - \FE_k(t) \widehat{\partial_v \mu}(\eta - kt)- i\sum_{l\in \ZZ} (\eta - k t) \FE_l(t) \widehat g_{k-l,\eta-lt}(t)
\end{equation}
for each $k\in \ZZ$ and $\eta \in \RR$. Thus, multiplying by $A_{k,\eta}$, we write 
\begin{equation}\label{eqs-g1} 
\begin{aligned}
\partial_t \widehat{Ag}_{k,\eta} 
&= - A_{k,\eta}\FE_k(t) \widehat{\partial_v \mu}(\eta - kt) - i\sum_{l\in \ZZ} (\eta - k t) \FE_l(t) \widehat{Ag}_{k-l,\eta-lt}(t)
\\
&\qquad - i\sum_{l\in \ZZ} (\eta - k t) [A_{k,\eta} - A_{k-l, \eta - lt}]\FE_l(t) \widehat g_{k-l,\eta-lt}(t) .
\end{aligned}
\end{equation} 
in which we note that, by symmetry $k\to k-l, l\to -l$ and $\FE_l=\overline{\FE_{-l}}$, the second term on the right does not contribute into \eqref{key-boundGg}, since 
\bea \label{div-free} &&\Re \sum_{k,l\in \ZZ} \int_{\RR} (\eta - k t) i\FE_l(t) \widehat{Ag}_{k-l,\eta-lt}(t) \overline{\widehat{Ag}_{k,\eta}(t)} \; d\eta\\ &&=\Re \sum_{k,l\in \ZZ} \int_{\RR} (\eta - (k-l) t) i\FE_{-l}(t) \widehat{Ag}_{k,\eta+lt}(t) \overline{\widehat{Ag}_{k-l,\eta}(t)} \; d\eta \notag\\ &&=-\Re \sum_{k,l\in \ZZ} \int_{\RR} (\eta - k t) \overline{i\FE_{l}}(t) \widehat{Ag}_{k,\eta}(t) \overline{\widehat{Ag}_{k-l,\eta-lt}(t)} \; d\eta= 0.\notag\eea 		
Let us estimate the first term coming from the right of \eqref{eqs-g1}. Using \eqref{basic-A} and the H\"older inequality, we have 
$$
\begin{aligned}
&\sum_{k\in \ZZ}  \int_\RR A_{k,\eta}|\FE_k(t) \widehat{\partial_v \mu}(\eta - kt) \widehat{Ag}_{k,\eta}(t)| d \eta
\\
& \le C_0\sum_{k\in \ZZ} A_{k,kt}|\FE_k(t)| 
\int_\RR e^{z\langle \eta - kt\rangle} |\widehat{\partial_v \mu}(\eta - kt)\widehat{Ag}_{k,\eta}(t)| \langle \eta-kt\rangle^{\sigma} d \eta
\\
& \le C_0 \sum_{k\in \ZZ} A_{k,kt}|\FE_k(t)| \Big(\int_\RR  |\widehat{Ag}_{k,\eta}(t)|^2\; d \eta\Big)^{1/2}
\\
& { \le C_0 \Big(\sum_{k\in \ZZ} A^2_{k,kt}|k|^{-2}|\Frho_k(t)|^2\Big)^{1/2} \Big(\sum_{k\in \ZZ}\int_\RR  |\widehat{Ag}_{k,\eta}(t)|^2\; d \eta\Big)^{1/2}}
\\&\le C_0 F[\rho(t)](z) G[g(t)]^{1/2}(z),
\end{aligned}
$$
{ upon recalling the definition \eqref{def-Nrho} of the norm $F[\cdot]$, $\FE_0(t) = 0$ and 
$
\FE_k(t) = \frac{1}{ik}\hat\rho_k(t).
$ 
}

Next, we estimate the last term coming from \eqref{eqs-g1}; namely, 
$$N_1: = \sum_{k,l\in \ZZ}  \int_\RR (\eta - k t) \frac{A_{k,\eta} - A_{k-l, \eta - lt}}{A_{l,lt}A_{k-l, \eta - lt}}\overline{\widehat{Ag}_{k,\eta}}(t) A_{l,lt}\FE_l(t) \widehat{Ag}_{k-l,\eta-lt}(t) \; d\eta.$$ 
We first bound the weight in the integral. We consider two cases: $\langle l, lt\rangle\ge \frac12 {\langle k, \eta\rangle}$ and $\langle k-l, \eta-lt\rangle\ge \frac12 {\langle k, \eta\rangle}$.

~\\
{\bf Case 1: $\langle l, lt\rangle\ge \frac12 {\langle k, \eta\rangle}$.} In this case, we bound $|A_{k,\eta} - A_{k-l, \eta - lt}| \le |A_{k,\eta}| + |A_{k-l, \eta - lt}|$, and use the fact that $|\eta - k t|\le \langle k,\eta\rangle \langle t\rangle$ and $|\eta - k t|\le \langle k-l, \eta - lt\rangle \langle t\rangle$. This yields 
$$
\begin{aligned}
\frac{|\eta - k t| |A_{k,\eta} - A_{k-l, \eta - lt}|}{A_{l,lt}A_{k-l, \eta - lt}} 
&\lesssim \langle k-l, \eta - lt\rangle^{-\sigma + 1} \langle t\rangle + \langle k, \eta \rangle \langle t\rangle \langle l, lt \rangle^{-\sigma} 
\\&\lesssim \langle t\rangle (\langle k-l\rangle^{-\sigma+1} + \langle l\rangle^{-\sigma + 1}), 
\end{aligned}
$$ 
upon recalling that $\sigma>2$ and $l\not =0$. 

~\\
{\bf Case 2: $\langle k-l, \eta-lt\rangle\ge \frac12 {\langle k, \eta\rangle}$.} In this case, we use following elementary inequality for $x\ge y$,  
\begin{equation} \label{algebra-ex} 
e^{\langle x\rangle^\gamma} \langle x\rangle^\sigma - e^{\langle y\rangle^\gamma} \langle y\rangle^\sigma \lesssim \frac{|x-y|}{\langle x\rangle^{1-\gamma} + \langle y\rangle^{1-\gamma}} e^{\langle x\rangle^\gamma}  \langle x\rangle^\sigma .
\end{equation}

Using \eqref{algebra-ex} and the definition of $A_{k,\eta}$, we bound 
$$
\begin{aligned}
|A_{k,\eta} - A_{k-l, \eta - lt}| 
&\le \frac{\langle l,lt\rangle }{\langle k,\eta\rangle^{1-\gamma} + \langle k-l, \eta - lt\rangle^{1-\gamma}} \Big[ A_{k,\eta} +A_{k-l,\eta - lt}\Big]. 
\end{aligned}$$ 
Using this and recalling that $\langle k-l, \eta-lt\rangle\ge \frac12 {\langle k, \eta\rangle}$, we obtain  
$$
\begin{aligned}
\frac{|\eta - k t| |A_{k,\eta} - A_{k-l, \eta - lt}|}{A_{l,lt}A_{k-l, \eta - lt}} 
&\lesssim  \frac{|\eta - kt|\langle l,lt\rangle^{-\sigma + 1} }{\langle k,\eta\rangle^{1-\gamma} + \langle k-l, \eta - lt\rangle^{1-\gamma}} .
\end{aligned}
$$
Note again that $|\eta - k t|\le \langle k,\eta\rangle \langle t\rangle$ and $|\eta - k t|\le \langle k-l, \eta - lt\rangle \langle t\rangle$. We thus get   
$$
\begin{aligned}
\frac{|\eta - k t| |A_{k,\eta} - A_{k-l, \eta - lt}|}{A_{l,lt}A_{k-l, \eta - lt}} 
&\lesssim  \langle k,\eta\rangle^{\gamma/2} \langle k-l, \eta - lt\rangle^{\gamma/2} \langle t\rangle \langle l, lt \rangle^{-\sigma + 1}
\\
&\lesssim  \langle k,\eta\rangle^{\gamma/2} \langle k-l, \eta - lt\rangle^{\gamma/2} \langle l\rangle^{-\sigma + 1}
\end{aligned}
$$ 
upon recalling that $\sigma>2$ and $l\not =0$.



Combining the above estimates, we thus have  
$$
\begin{aligned} 
N_1 &\le \langle t\rangle  \sum_{k,l\in \ZZ} (\langle k-l\rangle^{-\sigma+1} + \langle l\rangle^{-\sigma + 1} )
\\&\quad \times \int_\RR \langle k,\eta\rangle^{\gamma/2} \langle k-l, \eta - lt\rangle^{\gamma/2} 
|\widehat{Ag}_{k,\eta}(t) A_{l,lt}\FE_l(t) \widehat{Ag}_{k-l,\eta-lt}(t)| \; d\eta .
\end{aligned}$$ 
Recalling $\widehat{E}_l = \frac{1}{il}\widehat{\rho}_l$, { the definition \eqref{def-Nrho} of $F[\cdot]$,} 
and using the inequality $2ab \le a^2 + b^2$, we bound the integral in $\eta$ by 
$$
\begin{aligned}
& \int_\RR \langle k,\eta\rangle^{\gamma/2} \langle k-l, \eta - lt\rangle^{\gamma/2} 
|\widehat{Ag}_{k,\eta}(t) A_{l,lt}\widehat{E}_l(t) \widehat{Ag}_{k-l,\eta-lt}(t)| \; d\eta
\\&\le  F[\rho](t,z)
\int_\RR \Big[ \langle k,\eta\rangle^{\gamma} |\widehat{Ag}_{k,\eta}(t)|^2 
+ \langle l\rangle^{{2\alpha-2}} \langle k-l, \eta - lt\rangle^{\gamma} |\widehat{Ag}_{k-l,\eta-lt}(t)|^2\Big] \; d\eta .
\end{aligned}
$$
{ Note that $2\alpha-2<-1$, since $\alpha<1/2$.} Therefore, {since $\sigma>2$,} 
$$
\begin{aligned} 
N_1 &\le \langle t\rangle  F[\rho](t,z)
\sum_{k,l \in \ZZ} (\langle k-l\rangle^{-\sigma+1} + \langle l\rangle^{-\sigma + 1} )
\\&\quad \times\int_\RR \Big[ \langle k,\eta\rangle^{\gamma} 
|\widehat{Ag}_{k,\eta}(t)|^2 + \langle l\rangle^{{2\alpha-2}} 
\langle k-l, \eta - lt\rangle^{\gamma} |\widehat{Ag}_{k-l,\eta-lt}(t)|^2\Big] \; d\eta
\\
&\le C_0 \langle t\rangle F[\rho](t,z)  \partial_z G[g](t,z).
\end{aligned}
$$ 
This yields the desired bounds for the first term in \eqref{key-boundGg}.


Finally, the integral term involving $|\partial_t \partial_\eta \widehat g_{k,\eta} |$ is treated similarly, 
upon differentiating the equation (\ref{eqs-g}) in $\eta$. More precisely
\begin{equation}\label{eqs-gd} 
\begin{aligned}
\partial_t \partial_\eta \Fg_{k,\eta} 
&= - \FE_k(t) \partial_\eta \widehat{\partial_v \mu}(\eta - kt)
- i\sum_{l\in \ZZ} (\eta - k t) \FE_l(t) \partial_\eta \widehat g_{k-l,\eta-lt}(t)
\\&\quad - i\sum_{l\in \ZZ}  \FE_l(t) \widehat g_{k-l,\eta-lt}(t) .
\end{aligned}\end{equation}
The first two terms can be estimated exactly as in the previous case. It remains to bound the term involving the last term. 
Precisely, using \eqref{basic-A} and the H\"older inequality, we estimate 
$$
\begin{aligned} 
&N_2=\sum_{k,l \in \ZZ}  \int_{\RR} A_{k,\eta}^2 | \FE_l(t) \Fg_{k-l,\eta-lt}(t) \partial_\eta \Fg_{k,\eta}(t) | d\eta
\\&\le \sum_{l,k \in \ZZ}  \int_{\RR}  \frac{A_{k,\eta}}{A_{l,lt} A_{k-l,\eta - lt}} A_{l,lt}| \FE_l(t) | |l|^{-\gamma/2} 
\langle k-l\rangle^{\gamma/2} A_{k-l,\eta - lt}| \widehat g_{k-l,\eta-lt}(t) |
\\&\quad \times \langle k \rangle^{\gamma/2} A_{k,\eta}|\partial_\eta \Fg_{k,\eta}|\; d\eta
\\&\le F[\rho](t,z) 
\sum_{l,k \in \ZZ}\int_{\RR} \frac{A_{k,\eta}}{A_{l,lt} A_{k-l,\eta - lt}}  |l|^{\alpha-1-\gamma/2} 
\langle k-l\rangle^{\gamma/2} A_{k-l,\eta - lt}| \widehat g_{k-l,\eta-lt}(t) |
\\&\quad \times  \langle k \rangle^{\gamma/2} A_{k,\eta}|\partial_\eta \Fg_{k,\eta}|\; d\eta
\end{aligned} 
$$
As before, either $\langle l,l t\rangle\geq \frac 12 \langle k,\eta\rangle$ or $\langle k-l,k-l t\rangle \geq \frac 12 \langle k,\eta\rangle$. As a consequence,
$$
\frac{A_{k,\eta}}{A_{l,lt} A_{k-l,\eta - lt}} \lesssim  \langle l-k\rangle^{-\sigma} +
\langle l\rangle^{-\sigma}.
$$ 
Therefore,
\begin{align*}
&N_2\lesssim F[\rho](t,z) \sum_{l,k \in \ZZ}\Big (\langle l-k\rangle^{-\sigma} +
\langle l\rangle^{-\sigma}\Big)\\ &\qquad\times \Big(  |l|^{2\alpha-2-\gamma}\int_{\RR}
\langle k-l\rangle^\gamma A_{k-l,\eta - lt}^2| \widehat g_{k-l,\eta-lt}(t) |^2
d\eta  
+\int_{\RR} \langle k \rangle^\gamma A_{k,\eta}^2 |\partial_\eta \Fg_{k,\eta}|^2\; d\eta\Big)
\\&\qquad \le{ F[\rho](t,z)}  \partial_z G[g(t)](z)
\end{align*}
upon recalling that $\alpha<1/2$ and $\sigma>2$. 
The Proposition follows. 
\end{proof}


\subsection{A first bound on the electric field}


\begin{lem}\label{lem-gE} Define $F[\rho](t,z)$ as in \eqref{def-Nrho}. There holds  
$$
F[\rho](t,z) 
\le G[g(t)]^{1/2}(z)
$$
for any $z,t\ge 0$.
\end{lem}

The proof below applies to the case $d=1$. If $d > 1$, the argument can be trivially adapted by including higher order 
derivatives of $\hat g$ in $\eta$ up to order $d$.

\begin{proof}
From \eqref{new-rho}, we have 
$
\Frho_k(t) = \Fg_{k,kt}(t).
$
In view of \eqref{def-Nrho}, we need to bound $\Fg_{k,\eta}(t)$ pointwise in $\eta$. 
Observe that for any $k,\eta$ and $z\ge 0$, we have 
\begin{equation}\label{sup-g} 
\begin{aligned} 
| \Fg_{k,\eta}|^2 
&\le \int_{|\eta'| \ge |\eta|} |\Fg_{k,\eta'}||\partial_\eta \Fg_{k,\eta'}| \; d\eta' 
\le A_{k,\eta}^{-2}\int_\RR A_{k,\eta'}^2 |\Fg_{k,\eta'}|
| \partial_\eta \Fg_{k,\eta'}| \; d\eta'.
\end{aligned}\end{equation}
Hence,  
$$ 
\begin{aligned}
F[\rho](t,z) 
& \le  
\sup_{k\in \ZZ\setminus \{0\}}  A_{k,kt}|\Fg_{k,kt}| 
\\
& \le  
\Big(\sup_{k\in \ZZ\setminus \{0\}}  \int_\RR A_{k,\eta'}^2 |\Fg_{k,\eta'}|
| \partial_\eta \Fg_{k,\eta'}| \; d\eta'\Big)^{1/2}
\\
& \le  
\Big(\sum_{k\in \ZZ}  \int_\RR A_{k,\eta'}^2 |\Fg_{k,\eta'}|
| \partial_\eta \Fg_{k,\eta'}| \; d\eta'\Big)^{1/2}
\end{aligned}$$
which yields the Lemma. 	
\end{proof}


\subsection{Analyticity radius}


The radius of analyticity of our solutions will decrease with time. We denote it by $\lambda(t)$.
We now rewrite (\ref{FG1}) for a time-dependent $z$. For this, let
$$
\widetilde G[g](t,z) = G[g(t)](\lambda(t) z) , \qquad \widetilde F[\rho](t,z) =  F[\rho] (t,\lambda(t)z). 
$$
We have, for $z\in [0,1]$, 
\begin{equation}\label{FG1-lambda}
\begin{aligned}
\partial_t \widetilde G[g](t,z) - \Big[ \lambda'(t) z + C_0(1+t) \widetilde F[\rho](t,z) \Big] \partial_z \widetilde G[g](t,z) 
\le C_0\widetilde F[\rho](t,z) \widetilde G[g]^{1/2}(t,z)
\end{aligned}\end{equation}
Note that this equation is a transport like equation on $\widetilde G[g]$. 
If, for the domain $z\in [0,1]$,  the characteristics are outgoing at $z=0$ and $z=1$,
then $\widetilde G[g](t,z)$ for all positive $t$ and $z\in [0,1]$ may be bounded by its initial values at $t=0$. 
At $z = 0$, the characteristics are always outgoing
since the term in square brackets is manifestly positive. The characteristics are outgoing at $z = 1$ provided
\begin{equation}\label{cond-lambda}  
\lambda'(t) + C_0(1+t) F[\rho](t,\lambda(t)) \le 0 .
\end{equation}
We first choose  
\begin{equation}\label{def-lambda} 
\lambda(t) = \lambda_0 + \lambda_0 (1+t)^{-\delta}
\end{equation}
for some positive $\lambda_0$ such that
$\lambda_0 \le 1$ and $\lambda_0 \le \lambda_1 / 4$, and 
for an arbitrarily small positive constant $\delta \ll1$ so that 
\begin{equation}
\label{Gevrey-index} 3\gamma > 1 + 2\delta, 
\end{equation}
where $\gamma> 1/3$ is the index for the Gevrey-$\gamma$ regularity. Now if we assume that for all times
\begin{equation}\label{decay-rho0}
F[\rho](t,\lambda(t)) \le \sqrt \epsilon \langle t\rangle^{-\sigma + 1} 
\end{equation}
for $\epsilon$ sufficiently small, then (\ref{cond-lambda}) is satisfied for all times provided  $\sigma > 3 + \delta$.
Note that, at $t=0$, using Lemma \ref{lem-gE}, we get 
$$ 
F[\rho](0,\lambda_1) \le G[f^0]^{1/2}(\lambda_1).
$$
Hence, if 
$$
G[f^0](\lambda_1) \le \epsilon 2^{1 - \sigma} \le \epsilon,
$$ 
\eqref{cond-lambda} holds at $t=0$ and thus for sufficiently small times. 

Using (\ref{decay-rho0}) we can use (\ref{FG1-lambda}) in order to bound $G[g]$.
Namely, as long as \eqref{cond-lambda} is valid, the differential inequality \eqref{FG1-lambda} becomes,
after dividing the equation by  $\widetilde G[g]^{1/2}(t,z(t))$,
$$
\partial_t \widetilde G[g]^{1/2}(t,z(t))\le C_0 \sup_{0 \le z\le 1} \widetilde F[\rho](t,z) ,
$$ 
where $z(t)=z$ can be chosen arbitrarily 
in $[0,1]$ and the path $z(t')$ remains in the interval $[0,1]$ for all $t'\in [0,t]$.
This yields 
\begin{equation}
\label{bootstrap-g}
\begin{aligned} 
G[g(t)](\lambda(t)) &\le G[f^0](\lambda_1) + C_0 \left(\int_0^t F[\rho](s, \lambda(s)) \; ds \right)^2
\\&\le \epsilon +  C_0 \left(\int_0^t \sqrt\epsilon \langle s\rangle^{-\sigma + 1}\; ds\right)^2
\\&\le (1 + C_0)\epsilon, 
\end{aligned}\end{equation} 
for sufficiently small $\epsilon$.  

In the next section, we shall prove that under the bootstrap assumption \eqref{bootstrap-g}, for sufficiently small $\epsilon$, 
$F[\rho](t,\lambda(t))$ satisfies \eqref{decay-rho0}.
Theorem \ref{theo-LD} will then follow.


\subsection{Bounds on $\rho$}


This section is devoted to the bound on $F[\rho](t,\lambda(t))$. More precisely, we will prove
\begin{proposition}\label{prop-rho} 
Let $\sigma>3 + \delta$ and let $\lambda_0 < \lambda_1 / 4$.  Then provided $\eps$ is small enough,
\beq \label{assG}
\sup_{0\le t\le T}G[g](t,\lambda(t)) \le 2 \eps  (1 + C_0)
\eeq
implies  
\begin{equation}\label{decay-Frho} 
F[\rho](t,\lambda(t)) \le \sqrt \epsilon \langle t\rangle^{-\sigma + 1} 
\end{equation}
for $t\in [0,T]$, hence \eqref{cond-lambda} is valid for $t\in [0,T]$. 
\end{proposition}

This Proposition, combined with \ref{bootstrap-g} ends  the proof of Theorem \ref{theo-LD} using
a simple bootstrap argument.

\begin{proof} Integrating \eqref{eqs-g} in time and recalling that $
\Frho_k(t) = \Fg_{k,kt}(t),
$
we obtain 
\begin{equation}\label{eqs-rho} 
\Frho_k(t) + \int_0^t (t-s) \widehat{\mu}(k(t-s)) \Frho_k(s) \; ds = \FS_k(t)
\end{equation}
where the source term is defined by 
\begin{equation}
\label{def-FS}
\FS_k(t): =  \Ff^0_{k,kt}  - \sum_{l \not =0} \int_0^t \frac{k(t-s)}{l}\Frho_l(s) \Fg_{k-l,kt-ls}(s)\; ds .
\end{equation}
Hence, applying  Lemma \ref{lem-analyticrho}, we obtain 
	\begin{equation}\label{bd-propagator-re} 
	F[\rho](t,\lambda(t)) \le F[S](t,\lambda(t)) + C\int_0^t e^{-\frac14\theta_1 (t-s)}F[S](s,\lambda(s))\; ds 
	\end{equation}	
for $t\ge 0$. Note that in the above, we used that $\lambda(t)$ is decreasing and the monotonicity of the norm 
$F[S](s,z)$ in $z$. It thus suffices to give bounds on $F[S](t,\lambda(t))$. We have the following lemma. 

\begin{lemma}\label{lem-bdS}
Let $\gamma \in (1/3, 1]$ and $\sigma >3$. Under the assumption \eqref{assG}, there holds 
$$
\begin{aligned}
F[S](t,\lambda(t)) 
&\le e^{-\lambda_1 \langle t \rangle^\gamma / 2} G[f^0]^{1/2}(\lambda_1) 
+ C \sqrt \epsilon  \langle t\rangle^{-\sigma + 1} \sup_{0\le s\le t} F[\rho](s,\lambda(s)) \langle s\rangle^{\sigma-1}.
\end{aligned}
$$
\end{lemma}

Using Lemma \ref{lem-bdS},
 let us first end the proof of Proposition \ref{prop-rho}. Indeed, setting 
$$
 \zeta(t): = \sup_{0\le s\le t} F[\rho](s,\lambda(s)) \langle s\rangle^{\sigma-1}. 
 $$ 
we have 
$$\begin{aligned}
F[S](t,\lambda(t)) 
&\le e^{-\lambda_1 \langle t \rangle^\gamma / 2} G[f^0]^{1/2}(\lambda_1) 
+  C \sqrt \epsilon \zeta(t) \langle t\rangle^{-\sigma + 1}.
\end{aligned}
$$
Inserting this into \eqref{bd-propagator-re}, we get 
$$
\begin{aligned}
F[\rho](t,\lambda(t)) &\le e^{-\lambda_1 \langle t \rangle^\gamma / 2} G[f^0]^{1/2}(\lambda_1) 
+  C \sqrt \epsilon  \zeta(t) \langle t\rangle^{-\sigma + 1}
\\&\quad + C\int_0^t e^{-\frac14\theta_1 (t-s)}\Big[ e^{-\lambda_1 \langle s \rangle^\gamma / 2} G[f^0]^{1/2}(\lambda_1) 
+  C \sqrt\epsilon \zeta(s) \langle s\rangle^{-\sigma + 1}\Big]\; ds 
\\
&\le C_1 e^{-\lambda_1 \langle t \rangle^\gamma / 2} G[f^0]^{1/2}(\lambda_1) 
+ C_1\sqrt \epsilon \zeta(t) \langle t\rangle^{-\sigma + 1} ,
\end{aligned}$$
for some constant $C_1$. 
By definition of $\zeta(t)$, we thus have 
$$
\begin{aligned}
\zeta(t) &\le C_1 G[f^0]^{1/2}(\lambda_1) +
C_1 \sqrt \epsilon \zeta(t) .
\end{aligned} $$
Therefore, is $\epsilon$ is small enough,
$$
\zeta(t) \le 2 C_1 G[f^0]^{1/2} (\lambda_1),
$$
which is smaller than $\epsilon$, provided $G[f^0] (\lambda_1)$ is small enough.
This ends the proof of \eqref{decay-Frho}. 
\end{proof}

\begin{proof}[Proof of Lemma \ref{lem-bdS}]
From \eqref{def-FS}, for $k\not =0$, we compute 
$$ 
\begin{aligned}
&e^{ \lambda(t) \langle k, kt \rangle^\gamma } |\FS_k(t)| 
\langle k,kt\rangle^{\sigma} {|k|^{-\alpha}}
\\
& \le  e^{ \lambda(t) \langle k, kt \rangle^\gamma } |\Ff^0_{k,kt}  |
\langle k, kt\rangle^{\sigma}  {|k|^{-\alpha}}
\\&\quad + \sum_{l\in \ZZ\setminus \{0\}}  \int_0^t  e^{ \lambda(t) \langle k, kt \rangle^\gamma } 
\langle k, kt\rangle^{\sigma} {|k|^{1-\alpha}}(t-s)|l|^{-1} \Frho_l(s) \Fg_{k-l,kt-ls}(s)\; ds 
\\ & =: \mathcal{I}(t,k) + \mathcal{R}(t,k). 
\end{aligned}$$
Therefore, by definition, 
\begin{equation}
\label{est_SSS}
F[S](t,\lambda(t))  \le \sup_k \mathcal{I}(t,k) + \sup_k \mathcal{R}(t,k) .\end{equation}
Using Lemma \ref{lem-gE} for $z = \lambda_1$, the initial data term is bounded by 
$$
\begin{aligned} \mathcal{I}(t,k) 
& \le e^{- (\lambda_1 - \lambda(t)) \langle t \rangle^\gamma} G[f^0]^{1/2}(\lambda_1) 
\le  e^{- \lambda_1  \langle t \rangle^\gamma /2} G[f^0]^{1/2}(\lambda_1), 
\end{aligned}$$
for all $t\ge 0$, since $\lambda(t) \le 2 \lambda_0  < \lambda_1/2$. 

We next bound the reaction term $\mathcal{R}(t,k)$. Using the triangle inequality 
$$
\langle k,kt\rangle \le \langle l,ls\rangle + \langle k-l,kt-ls\rangle,
$$ 
we have  
$$  
e^{ \lambda(t) \langle k, kt \rangle^\gamma }
\le e^{[\lambda(t) - \lambda(s) ]\langle k, kt\rangle^\gamma} e^{\lambda(s)\langle l,ls\rangle^\gamma} e^{\lambda(s) \langle k-l, kt-ls\rangle^\gamma }  .
$$
Using the pointwise estimate \eqref{sup-g} with $z = \lambda(s)$ 
and distributing the corresponding weight for $\Frho_l(s)$ and $\Fg_{k-l,\eta}(s)$, we bound 
$$ \begin{aligned}
\mathcal{R}(t,k) 
& \le
\sum_{l\not =0} \int_0^t {|k|^{1-\alpha}}(t-s)|l|^{-1}  e^{(\lambda(t)-\lambda(s)) \langle k,kt\rangle^\gamma } \langle k, kt\rangle^{\sigma} 
\\& \qquad \times  e^{\lambda(s)\langle l,ls\rangle^\gamma } |\Frho_l(s)| e^{\lambda(s) \langle k-l, kt-ls\rangle^\gamma } |\Fg_{k-l,kt-ls}(s) | \; ds 
\\
& \le
\sum_{l\not =0} \int_0^t C_{k,l}(t,s)F[\rho](s,\lambda(s)) 
\\
&\quad \times \Big(\int_{\RR} e^{2\lambda(s) \langle k-l, \eta\rangle^\gamma } | \Fg_{k-l,\eta}(s) \partial_\eta \Fg_{k-l,\eta}(s) | \langle k-l, \eta \rangle^{2\sigma} \; d\eta\Big)^{1/2} ds
\\
& \le
C_0 \sqrt\epsilon \sum_{l\not =0} \int_0^t C_{k,l}(t,s) \langle s\rangle^{-\sigma + 1} \; ds \sup_{0\le s\le t} F[\rho](s,\lambda(s)) \langle s\rangle^{\sigma - 1},
\end{aligned}$$
upon using the assumption \eqref{assG}, where 
$$ 
C_{k,l}(t,s): = {|k|^{1-\alpha}}(t-s) {|l|^{\alpha-1}}  e^{(\lambda(t)-\lambda(s)) \langle k,kt\rangle^\gamma} 
\langle k, kt\rangle^{\sigma} \langle k-l, kt - ls \rangle^{-\sigma} \langle l, ls \rangle^{-\sigma} $$
{ for $\alpha <\frac12$ as in \eqref{def-Nrho}.}
Lemma \ref{lem-bdS} thus follows from the following claim 
\begin{equation}\label{claimCR}
\sup_{k\not =0}\sum_{l\not =0} \int_0^t C_{k,l}(t,s) \langle s\rangle^{-\sigma + 1} \;ds \le C_0 \langle t\rangle^{-\sigma + 1},
\end{equation}
uniformly in $t\ge 0$, which we shall now prove. 
Note that $ \langle l,ls\rangle + \langle k-l,kt-ls\rangle \ge \langle k,kt\rangle$. 
Hence, either $\langle l, ls\rangle$ or $ \langle k-l, kt-ls\rangle$ is greater than $\langle k, kt\rangle/2$. 
Let us consider each case separately.


\subsubsection*{Case 1: $\langle k-l,  kt - ls \rangle \ge \frac12 \langle k, kt\rangle$.} 


In this case, we have 
$$
C_{k,l}(t,s) \le {|k|^{1-\alpha}}(t-s) {|l|^{\alpha-1}}  e^{(\lambda(t)-\lambda(s)) \langle k,kt\rangle^\gamma} 
\langle l, ls \rangle^{-\sigma} .
$$
Recall that $k,l\not =0$. We consider two cases: $s\le t/2$ and $s\ge t/2$.
In the former case, we have
$$
 \lambda(s) - \lambda(t) = \frac{\lambda_0}{\langle s\rangle^\delta} - \frac{\lambda_0}{\langle t\rangle^\delta} 
 \ge  \lambda_0(2^\delta -1) \langle t\rangle^{-\delta} \ge \theta_\delta \langle t\rangle^{-\delta}
 $$ 
	for some positive constant $\theta_\delta$. Therefore, we have  
\begin{equation}\label{smalls}
\begin{aligned} 
C_{k,l}(t,s) &\le {|k|^{1-\alpha}}(t-s){|l|^{\alpha-1}} e^{ - \theta_\delta \langle k\rangle^{\gamma-\delta}\langle t\rangle^{\gamma-\delta}}  
\langle l, ls \rangle^{-\sigma} 
\le e^{ - \theta'_\delta \langle t\rangle^{\gamma-\delta}}  
\langle l\rangle^{-\sigma}
\end{aligned} \end{equation}
which clearly implies \eqref{claimCR}. On the other hand, for $s\ge t/2$, we use $|\lambda(s) - \lambda(t)|\ge \theta_0|t-s|/t^{1 + \delta}$ to bound 
$$
\begin{aligned}
C_{k,l}(t,s) 
&\le {|k|^{1-\alpha}}(t-s){|l|^{\alpha-1}}  e^{-\theta_0 |k^\gamma(t-s)| / t^{1-\gamma+\delta}} 
\langle l, ls \rangle^{-\sigma} 
\\&\le \frac{C_0 t^{\frac{1-\gamma+\delta}{\gamma}{(1-\alpha)}} |l|^{-1+\alpha}  \langle l, l t \rangle^{-\sigma}}{|t-s|^{{\frac{1-\alpha}{\gamma}}-1}},
\end{aligned}$$  
which again gives \eqref{claimCR}, { noting that $\frac{1-\alpha}\gamma -1 <1$ for $\gamma \in (\frac13,1]$,$\alpha \in (\frac13,\frac12)$ and $\sigma>3$.} 


\subsubsection*{Case 2: $\langle l, ls \rangle \ge \frac12 \langle k, kt\rangle$.}  In this case, we have 


$$ 
C_{k,l}(t,s)\le  {|k|^{1-\alpha}} (t-s) {|l|^{\alpha-1}}  e^{(\lambda(t)-\lambda(s)) \langle k,kt\rangle^\gamma } \langle k-l,  kt - ls \rangle^{-\sigma} .
$$
As in the previous case, we consider two cases: 
$|kt - ls|\le t/4$ and $|kt - ls| \ge t/4$. For the latter, we argue as in the regime $s\ge t/2$ from the previous case and easily obtain \eqref{claimCR}. 

We thus focus on the regime $|kt - ls|\le t/4$. 
Note that when $k=l$, the claim \eqref{claimCR} clearly holds thanks to the $\langle t-s\rangle^{-\sigma}$ decay, while for $s\le t/2$, the claim follows from having the exponential decay $ e^{ - \theta_\delta \langle k\rangle^{\gamma-\delta}\langle t\rangle^{\gamma-\delta}}  $, similar to \eqref{smalls}. For $k\not =l$ and $s\ge t/2$, we note 
$$ |k(t-s)| = |kt - ls + (l-k) s| \ge |l-k| s - |kt - ls| \ge \frac14 |l-k| t. $$
Thus, using $|\lambda(s) - \lambda(t)|\ge \theta_0|t-s|/t^{1 + \delta}$, we have 
\begin{equation}\label{radius-slow} e^{(\lambda(t) - \lambda(s)) \langle k,kt\rangle^\gamma } 
\le e^{-\theta_0 \langle t\rangle^{-1-\delta}|t-s| \langle k,kt\rangle^\gamma }  \le e^{ - \theta_\delta \langle t\rangle^{\gamma-\delta} \frac{|l-k|}{|k|^{1-\gamma}}} . \end{equation}
Clearly, for $|l-k|\ge |k|/2$, the above yields an exponential decay in time, and \eqref{claimCR} follows. It remains to consider the case when $|l-k|\le |k|/2$, which in particular gives $|l|\ge |k|/2$. We have 
$$ 
\begin{aligned}
C_{k,l}(t,s) 
&\le  (t-s)  e^{(\lambda(t)-\lambda(s)) \langle k,kt\rangle^\gamma } \langle k-l,  kt - ls \rangle^{-\sigma} 
\\
&\le \frac{|l-k| t}{l^2} e^{ - \theta_\delta \langle t\rangle^{\gamma-\delta} \frac{|l-k|}{|k|^{1-\gamma}}}  \langle k-l,  kt - ls \rangle^{-\sigma} |l|
\\
&\lesssim \frac{|l-k|^{1 - \frac{1}{\gamma-\delta}}}{l^{2} |k|^{\frac{\gamma-1}{\gamma-\delta}}} \langle k-l,  kt - ls \rangle^{-\sigma} |l|  
\\
&\lesssim |k|^{-2 + \frac{1-\gamma}{\gamma-\delta}}\langle k-l, kt - ls \rangle^{-\sigma} |l| 
\\
& \lesssim \langle k-l,  kt - ls \rangle^{-\sigma} |l|
\end{aligned}
$$
where the last estimate has used $3\gamma > 1 + 2\delta$, giving the limitation of Gevrey-$\gamma$ data, with $\gamma>1/3$. Thus, we bound 
$$
\begin{aligned} 
\int_{t/2}^t C_{k,l}(t,s) \langle s\rangle^{-\sigma + 1}\;ds 
&\le C_0 \langle t\rangle^{-\sigma + 1}   \int_{t/2}^t \langle k-l,  kt - ls \rangle^{-\sigma} |l|\; ds  
\\
&\le C_0 \langle k-l\rangle^{-2} \langle t\rangle^{-\sigma + 1} .
\end{aligned}
$$
The Proposition follows. 
\end{proof}


\end{document}